\documentclass[12pt,reqno,a4paper]{amsart}

\usepackage[T1]{fontenc}
\usepackage[cmintegrals,cmbraces]{newtxmath}
\usepackage{ebgaramond}
\usepackage{ebgaramond-maths}
\usepackage{euscript}
\makeatletter
  \DeclareSymbolFont{ntxletters}{OML}{ntxmi}{m}{it}
  \SetSymbolFont{ntxletters}{bold}{OML}{ntxmi}{b}{it}
  \re@DeclareMathSymbol{\partial}{\mathord}{ntxletters}{"40}
\makeatother

\usepackage{tikz}
\usepackage{color}
\definecolor{green}{RGB}{0,127,0}
\definecolor{red}{RGB}{191,0,0}
\usepackage[colorlinks]{hyperref}
\usepackage{hyperref}

\makeatletter
\def\@settitle{\baselineskip14\p@\relax
    \flushleft{\Large\bfseries \@title}
}
\def\@@and{\&}
\def\@setauthors{%
  \begingroup
  \def\thanks{\protect\thanks@warning}%
  \trivlist
  \centering\footnotesize \@topsep30\p@\relax
  \advance\@topsep by -\baselineskip
  \item\relax
  \author@andify\authors
  \def\\{\protect\linebreak}%
  \flushleft\itshape{by \authors}%
\relax\vskip2\baselineskip%
  \centerline{\swshape Dedicated to Jacek Świątkowski on the occasion of his 60th birthday}%
\ifx\@empty\contribs
  \else
    ,\penalty-3 \space \@setcontribs
    \@closetoccontribs
  \fi
  \endtrivlist
  \endgroup
}
\renewcommand{\uppercasenonmath}[1]{\relax}
\def\@secnumfont{\bfseries}
\def\specialsection{\@startsection{section}{1}%
  \z@{\linespacing\@plus\linespacing}{.5\linespacing}%
  {\normalfont\bfseries\raggedright\boldmath}}
\def\section{\@startsection{section}{1}%
  \z@{.7\linespacing\@plus\linespacing}{.5\linespacing}%
  {\normalfont\bfseries\raggedright\boldmath}}
\def\ps@headings{\ps@empty
  \def\@evenhead{%
    \setTrue{runhead}%
    \normalfont\scriptsize
    \llap{\thepage\ }\hrulefill\ {\itshape\shorttitle}
    \def\thanks{\protect\thanks@warning}}%
  \def\@oddhead{%
    \setTrue{runhead}%
    \normalfont\scriptsize
    \def\thanks{\protect\thanks@warning}%
    {\itshape\authors}\ \hrulefill\rlap{\ \thepage}}%
  \let\@mkboth\markboth
}
\makeatother

\renewcommand{\emph}[1]{{\bfseries\boldmath #1}}

\newtheorem{theorem}[subsection]{Theorem}
\newtheorem{corollary}[subsection]{Corollary}
\newtheorem{lemma}[subsection]{Lemma}
\newtheorem{proposition}[subsection]{Proposition}

\theoremstyle{definition}

\theoremstyle{remark}
\newtheorem{remarkx}[subsection]{Remark}
\newenvironment{remark}
    {\pushQED{\qed}\remarkx}
    {\popQED\endremarkx}

\numberwithin{equation}{section}
\numberwithin{figure}{section}

\newcommand{\B}[1]{{\mathbf #1}}

\newcommand{\F}[1]{{\mathfrak #1}}
\newcommand{\R}{\mathrm}

\begin{document}

\title{A refinement of bounded cohomology}
\author{Światosław Gal}
\author{Jarek Kędra}



\begin{abstract}
We introduce a refinement of bounded cohomology and prove that the suitable
comparison homomorphisms vanish for an amenable group.
We investigate in this context Thompson's group
$F$ and provide further evidence towards its amenability.
We show that the space
of 1-bounded cocycles of degree two is essentially as big as the space of
Lipschitz functions on the underlying group.  We also explain that such classes
define metrics on the group.  
\end{abstract}

\maketitle

\section{Introduction}\label{S:intro}

Although the cohomology of a space is a homotopy invariant, it sometimes
captures geometric information. For example, existence of a K\"ahler metric on
a $2n$-dimensional manifold $M$ implies that the dimension of the odd degree
cohomology of $M$ is even \cite{zbMATH03634395}. Another impressive example is
simplicial volume, which is a cohomological counterpart of the minimal volume
of a Riemannian manifold. Here, however, one has to put more structure on the
space of cochains.  Namely, consider those that are bounded functions on the
spaces of singular chains. This has been done by Gromov in
\cite{zbMATH03816552} where he introduced the {\em bounded cohomology}. Since
then it became a standard tool in various branches of mathematics like group
theory, differential geometry, dynamical systems and
others~\cite{zbMATH01614276}. Gromov observed that the bounded cohomology of a
space depends only on its fundamental group which, in principle, reduced the
theory to group cohomology. 

It is well known that the bounded cohomology of an amenable group is trivial
\cite[Section 3]{zbMATH03816552}. If one looks closely at the proof of this
fact \cite{zbMATH06821283,zbMATH05304367}, it is clear that the standard
averaging argument yields more. Not only bounded cocycles become coboundaries
but also cocycles that are bounded after fixing a number of variables do so.
This observation (presented in the proof of Theorem \ref{T:comparison} below)
led us to defining semibounded cohomology which is the main subject of this
paper. 

We investigate the properties of semibounded cohomology in relation to other
geometric properties of spaces or groups. For example, we show that a
cohomology class of a group is hyperbolic if and only if it admits a
$1$-bounded representative; see Section \ref{S:hyp classes} for definitions and
more details.

It is well known that the kernel of the comparison map $H^2_b(\Gamma;\B R)\to
H^2(\Gamma;\B R)$ consists of quasimorphisms and is often an infinite
dimensional space. This holds, for example, for non-elementary hyperbolic groups, mapping
class group which are not virtually abelian and many others \cite{zbMATH01928904,zbMATH05658119}.
Quasimorphisms on a group $\Gamma$  form a function theory
tightly related to the geometry of the stable commutator length on the commutator
subgroup $[\Gamma,\Gamma]$; see \cite{zbMATH05577332} for more details. In this
spirit we show in Section \ref{S:abstractly-lipschitz} that the kernel of the
comparison map between 1-bounded degree 2 cohomology and the standard
cohomology consist of functions which we call abstractly Lipschitz. These are
functions for which there exists a length on $\Gamma$ with respect to which the
function is Lipschitz.

A similar concept to 1-bounded cocycles was defined by Neumann and
Reeves~\cite{zbMATH00991907} for 2-cocycles and by Frigerio and
Sisto~\cite{zbMATH07734709} for any degree under the name {\itshape weakly bounded}
cocycles.  Both papers consider cocycles that take finite set of values
after fixing all but the first argument.

Weakly bounded 2-cocycles were used to characterise central extensions of
finitely generated groups that are quasi-isometric with trivial extensions.
It was a folklore that this holds if the defining cocycle for an
extension is bounded.  Necessary and sufficient condition in terms of
$\ell^\infty$-cohomology (see \cite{arxiv.2107.09089}) was given by
Kleiner and Leeb \cite[Proposition 8.3]{zbMATH01652733}.  The characterisation in
terms of weakly bounded cocycles is due to Frigerio and Sisto \cite[Corollary
2.5]{zbMATH07734709}.
The connection between weakly bounded cocycles and $\ell^\infty$-cohomology
is a direct analogy to our Theorem~\ref{T:iso-bockstein}.

It is a longstanding open problem (popularised by R. Geoghegan around 1979) to
determine whether Thopson's group $F$ is amenable
\footnote{
A short historical account is given here: 
\url{https://mathoverflow.net/questions/55214/does-the-amenability-problem-for-thompsons-group-f-predate-1980}
}.
It is a recent result of Monod that the bounded cohomology of $F$ vanishes
\cite{zbMATH07536483} and in Section \ref{S:test} we apply our refined theory
to test the amenability of $F$. Unfortunately, we prove that no class in the
standard cohomology of $F$ can be represented by a 1-bounded cocycle.  This
shows that the comparison homomorphisms between semibounded and the standard
cohomology are trivial and, as in the case of Monod's vanishing result, provides
inconclusive answer as to whether $F$ is amenable. The remaining line of attack
within our framework is to investigate the comparison homomorphisms between
various semibounded cohomology of $F$.  However, they seem to be very hard to
compute.

\bigskip

Throughout the paper when  we consider group cohomology we use nonhomogeneous
notation for chains \cite{zbMATH05056625} with exceptions in Section \ref{S:hyp
classes}.


\section{Semibounded cohomology}\label{S:semi-bdd-coho}
Let $\Gamma$ be a discrete group.
An $n$-cochain
$\omega \colon \Gamma^n\to \B R$ is called \emph{$p$-bounded} if for fixed
$g_{p+1},\dots,g_{n}\in \Gamma$ the function
$$
\Gamma^p\ni g_1,\dots,g_p \mapsto \omega(g_1,\dots,g_p,g_{p+1},\dots,g_n)
$$
is bounded. For example, an $n$-bounded cochain is bounded in the usual sense.
It is straightforward to check that
$C^*_{(p)}(\Gamma;\B R)\subseteq C^*(\Gamma;\B R)$
consisting of $p$-bounded cochains is preserved by the codifferential
and hence it is a subcomplex. Indeed, if 
$\alpha \in C^{n}_{(p)}(\Gamma;\B R)$ is a $p$-bounded $n$-chain and
$g_{p+1},\ldots,g_{n+1}\in \Gamma$ are fixed elements then
the following computation shows that $\delta\alpha$ is $p$-bounded.
\begin{align*}
\delta\alpha(g_1,\ldots,g_{n+1})
&=\alpha(g_2,\ldots,g_{p+1},\textcolor{blue}{g_{p+2},\ldots,g_{n+1}})\\
&\hphantom{=}+ \sum_{k=1}^{p} (-1)^k \alpha(g_1,\ldots,g_kg_{k+1},\ldots,g_{p+1},
\textcolor{blue}{g_{p+2},\ldots,g_{n+1}})\\
&\hphantom{=}+ \sum_{k=p+1}^{n} (-1)^k \alpha(g_1,\ldots,g_p,
\textcolor{blue}{g_{p+1},\ldots,g_kg_{k+1},\ldots,g_{n+1}})\\
&\hphantom{=}
+(-1)^{n+1} \alpha(g_1,\ldots,g_p,\textcolor{blue}{g_{p+1},\ldots,g_n}).
\end{align*}

The cohomology of the complex $C^*_{(p)}(\Gamma;\B R)$ 
will be denoted by $H^n_{(p)}(\Gamma;\B R)$.
The inclusions $C^*_{(p)}(\Gamma;\B R)\subseteq C^*_{(p-1)}(\Gamma;\B R)$
induce maps in cohomology 
$$
H^n_{(n)}(\Gamma;\B R)\to H^n_{(n-1)}(\Gamma;\B R)
\to \cdots \to H^n_{(1)}(\Gamma;\B R) \to H^n_{(0)}(\Gamma;\B R)=H^n(\Gamma;\B R),
$$
called  \emph{comparison homomorphisms}. 
Notice that $H^n_{(n)}(\Gamma;\B R)=H^n_b(\Gamma;\B R)$.
In what follows, we will refer to $H^n(\Gamma;\B R)$ as the standard cohomology. 

Recall that a group $\Gamma$ is called \emph{amenable} if 
it admits a left-invariant \emph{mean}. That is,
a linear functional $\ell^{\infty}\Gamma\to \B R$, on the left module
of bounded functions on $\Gamma$, such that
\begin{enumerate}
\item $\int F(gh)\F m(g) = \int F(g)\F m(g)$ for all $h\in \Gamma$ and 
all $F\in \ell^{\infty}\Gamma$;
\item $\int 1\F m(g) = 1$;
\item if $F\geq 0$ almost everywhere then $\int F(g)\F m(g)\geq 0$.
\end{enumerate}

It is well known \cite[Section 3.0]{zbMATH03816552} that if a group $\Gamma$ is
amenable then its bounded cohomology vanishes in positive degrees.  In
particular, the comparison homomorphisms between bounded and standard
cohomology are trivial. We generalise this result to the $p$-bounded
cohomology.

\begin{theorem}\label{T:comparison}
If\/ $\Gamma$ is an amenable group then the comparison homomorphisms
$$
H^n_{(p)}(\Gamma;\B R)\to H^n_{(p-1)}(\Gamma;\B R)
$$
vanish for all $n>0$ and $p=1,2,\dots, n$.
\end{theorem}

\begin{proof}
The argument follows the proof of the fact that bounded cohomology
of an amenable group vanishes in positive degree.
Let $\F m$ be a right-invariant mean on $\Gamma$.  Define a map
$M \colon C^{n+1}_{(p)}(\Gamma;\B R)\to C^n_{(p-1)}(\Gamma;\B R)$ by
$$
M\omega(g_1,\dots,g_n)=\int\omega(h,g_1,\dots,g_n)\,\F m(h).
$$
To see that $M\omega$ is $(p-1)$-bounded fix $g_p,\ldots,g_n$ and
observe that the function
$h\mapsto \omega(h,g_1,\ldots,g_{p-1},g_p,\ldots,g_n)$
is bounded uniformly for any choice of $g_1,\ldots,g_{p-1}$. 
This shows $(p-1)$-boundedness of $M\omega$.

Then
\begin{align*}
M\delta\omega(g_1,\dots,g_n)
&=\int\delta\omega(h,g_1,\dots,g_n)\,\F m(h)\\
&=\int\omega(g_1,\dots,g_n)\F m(h)\\
&\hphantom{=}-\int\omega(hg_1,g_2,\dots,g_n)\,\F m(h)+\int\omega(h,g_1g_2,\dots,g_n)\,\F m(h)-\dots\\
&\hphantom{=}\dots\pm \int\omega(h,g_1,\ldots,g_{n-1})\,\F m(h)\\
&\overset{!}=\omega(g_1,\dots,g_n)\\
&\hphantom{=}-\int\omega(h,g_2,\dots,g_n)\,\F m(h)+\int\omega(h,g_1g_2,\dots,g_n)\,\F m(h)-\dots\\
&\hphantom{=}\dots\pm \int\omega(h,g_1,\ldots,g_{n-1})\,\F m(h)\\
&=\omega(g_1,\dots,g_n)-M\omega(g_2,\dots,g_n)+M\omega(g_1g_2,\dots,g_n)-\dots\\
&\hphantom{=}\dots\pm M\omega(g_1,\ldots,g_{n-1})\\
&=\omega(g_1,\dots,g_n)-\delta M\omega(g_1,\dots,g_n),
\end{align*}
where, in the marked equality, we used the invariance of the mean 
in the second term.
We obtain that
$$
M\delta\omega+\delta M\omega=\omega=i(\omega),
$$
where $i\colon C^n_{(p)}(\Gamma;\B R)\to C^n_{(p-1)}(\Gamma;\B R)$ is the inclusion
of $p$-bounded cochains into $(p-1)$-bounded ones.
Consequently, the map $M$ is a homotopy for the comparison map.
\end{proof}


\begin{remark}
Notice that we only prove the vanishing of the comparison maps. We don't know
whether $p$-bounded cohomology of an amenable group vanishes for $p\neq n$.
\end{remark}

\begin{remark}  One could use another definition of $p$-boundedness.
Namely, an $n$-cochain
$\omega \colon \Gamma^n\to \B R$ is called \emph{left-$p$-bounded} if for fixed
$g_1,\dots,g_{n-p}\in \Gamma$ the function
$$
\Gamma^p\ni g_{n-p+1},\dots,g_n \mapsto \omega(g_1,\dots,g_{p-n},g_{p-n+1},\dots,g_n)
$$
is bounded.  Such definition (of 1-bounded 2-cocycle) appears in \cite[Section 4]{zbMATH06065373}
under the name of \emph{semibounded} cocycle, where the authors use it
to prove Polterovich's theorem stating that
on the group of symplectic diffeomorphisms of a symplectically hyperbolic
manifold all its elements are undistorted \cite[Theorem 5.2]{zbMATH06065373}.

Nevertheless, the involution
$$\widehat\omega(g_1,\dots,g_n)=\omega(g_n^{-1},\dots,g_1^{-1})
$$
is a chain map that interchanges $p$-bounded and left-$p$-bounded cocycles
(with trivial coefficients), so both notions define isomorphic objects.
\end{remark}


\section{The 1-bounded cohomology}\label{L:last}

Let $\ell^\infty\Gamma$ denote the left module of bounded functions on $\Gamma$
with the action defined by $(hF)(g)=F(gh)$.

\begin{lemma}\label{L:lgamma}
The map $\Lambda_0\colon C^n_{(1)}(\Gamma;\B R)\to C^{n-1}(\Gamma;\ell^\infty\Gamma)$
given by the formula
$$
\Lambda_0\omega(g_1,\dots,g_n)(h)=\omega(h,g_1,\dots,g_n)
$$
is a homotopy, \textit{i.e.} $\delta\Lambda_0\omega+\Lambda_0\delta\omega=\omega$.
\end{lemma}

\begin{proof}
\begin{align*}
\Lambda_0\delta\omega(g_0,\dots,g_n)(h)&=\delta\omega(h,g_0,\dots,g_n)\\
&=\omega(g_0,\dots,g_n)\\
&\hphantom{=}-\omega(hg_0,\dots,g_n)+\omega(h,g_0g_1,\dots,g_n)\pm\dots\\
&=\omega(g_0,\dots,g_n)\\
&\hphantom{=}-\Lambda_0\omega(g_1,\dots,g_n)(hg_0)+\Lambda_0\omega(g_0g_1,\dots,g_n)(h)\pm\dots\\
&=\omega(g_0,\dots,g_n)-\delta\Lambda_0\omega(g_0,\dots,g_n)(h).
\end{align*}
\end{proof}

Let $\Lambda\colon C^n_{(1)}(\Gamma;\B R)\to C^{n-1}(\Gamma;\ell^\infty\Gamma/\B R)$
be the composition of $\Lambda_0$ followed by the map
$C^{n-1}(\Gamma;\ell^{\infty}\Gamma) \to C^{n-1}(\Gamma;\ell^{\infty}\Gamma/\B R)$
induced by the quotient of the coefficients.
The map $\Lambda$ is surjective and anticommutes with the codifferential.  
Its kernel consists of cocycles that do not depend on the first variable.
We call them \emph{1-constant} and denote by $C^n_{[1]}(\Gamma,\B R)$.

\begin{lemma}\label{L:1-const}
The complex of 1-constant cochains with restricted codifferential is acyclic.
\end{lemma}

\begin{proof}
We need to show that $\delta$ preserves $1$-constant cochains. That is,
if $\Lambda\omega=0$ then $\Lambda\delta\omega=0$. This is equivalent to
showing that if for any $g'_1,\ldots,g'_{n-1}\in \Gamma$ the function
$\Lambda_0\omega(g'_1,\ldots,g'_{n-1})$ is constant then so is the function
$\Lambda_0\delta\omega(g_1,\ldots,g_{n})$ for any 
$g_1,\ldots,g_{n+1}\in \Gamma$. According to Lemma \ref{L:lgamma}, we have
$\Lambda_0\delta\omega = \omega - \delta\Lambda_0\omega$. 
Evaluating both sides on $g_1,\ldots,g_n$ we get that both
terms on the right hand side are constant, which proves that $1$-constant
cochains form a subcomplex.

The map $H$ defined by
$$
H\omega(g_2,\dots,g_n)=\omega(1,1,g_3,\dots,g_n)
$$
is a homotopy proving acyclicity. Indeed, we have
$$
H\delta\omega(g_1,\ldots,g_n) + \delta H\omega(g_1,\ldots,g_n)
=\omega(1,g_2,\ldots,g_n).
$$
Since $h\mapsto \omega(h,g_2,\ldots,g_n)$ is a constant function,
the right hand side of the above equality is equal to $\omega(g_1,\ldots,g_n)$.
\end{proof}

Consider the following short exact sequence
$$
0\to C^n_{[1]}(\Gamma,\B R)\to C^n_{(1)}(\Gamma;\B R)\to C^{n-1}(\Gamma;\ell^\infty\Gamma/\B R)\to0
$$
and the induced long exact sequence in cohomology.  As a corollary, we
deduce the following theorem.

\begin{theorem}\label{T:iso-bockstein}
The map induced by $\Lambda$,
$$
H^n_{(1)}(\Gamma;\B R)\to H^{n-1}(\Gamma;\ell^\infty\Gamma/\B R)
$$
is an isomorphism.
Its composition with the Bockstein homomorphism
$$
H^{n-1}(\Gamma;\ell^\infty\Gamma/\B R)\to H^n(\Gamma;\B R)
$$
corresponding to the short exact sequence of coefficients
$\B R\to \ell^\infty\Gamma\to\ell^\infty\Gamma/\B R$
is equal to the comparison map $H^n_{(1)}(\Gamma;\B R)\to H^n(\Gamma;\B R)$.
\end{theorem}

\begin{proof}
Let $\omega \in C^n_{(1)}(\Gamma;\B R)$ be a cocycle. 
Then $\Lambda_0(\omega)\in C^{n-1}(\Gamma;\ell^{\infty}\Gamma)$ is
the lift of $\Lambda(\omega)\in C^{n-1}(\Gamma;\ell^{\infty}\Gamma/\B R)$.
By definition \cite[Section 3.E]{zbMATH02103273}, the Bockstein homomorphism sends $[\Lambda(\omega)]$ to
$[\delta\Lambda_0(\omega)]$.  Since $\omega$ is a cocycle, it follows from
Lemma \ref{L:lgamma} that 
$\delta\Lambda_0(\omega) = \omega - \Lambda_0(\delta\omega) = \omega$,
which proves the statement.
\end{proof}

Notice that a group $\Gamma$ is amenable if and only if the short exact
sequence of coefficients $0\to\B R\to\ell^\infty\Gamma\to\ell^\infty\Gamma/\B R\to 0$ splits.
Indeed, if $s\colon \ell^{\infty}\Gamma/\B R \to \ell^{\infty}\Gamma$ is a section then
$F\mapsto F - s(F\cdot \B R)$ is a mean. Conversely, if there exists a mean then
$s[F] = F - \int F(g)\F m(g)$ is a well defined section.
In such a case the Bockstein map is trivial in accordance with Theorem \ref{T:comparison}.


\section{1-bounded 2-cocycles}\label{S:abstractly-lipschitz}

In this section we discuss two constructions of pseudometrics on a group
defined by 1-bounded 2-cocycles (in what follows, we will abuse terminology
and refer to them as metrics).  The first metric is defined by cocycles that
are trivial in the standard cohomology.  The second construction is defined for
arbitrary cocycles.

We start with the first construction.
Recall, that any bounded 2-cocycle in the kernel of the comparison homomorphism
$H^2_b(\Gamma;\B R)\to H^2(\Gamma;\B R)$ is a codifferential of a
quasimorphism, \textit{i.e.}, a function $\phi\colon\Gamma\to\B R$ such that
$\sup_{g,h\in\Gamma}|\phi(g)-\phi(gh)+\phi(h)|<\infty$.  
Thus the kernel of the comparison map can be identified with the space of 
quasimorphisms divided by the sum of the space of homomorphisms 
(annihilated by the codifferential) and the space of bounded functions
(which correspond to coboundaries). See Calegari \cite[Chapter 2]{zbMATH05577332} 
for a survey.

Similarly, any 1-bounded 2-cocycle in the kernel of the comparison homomorphism
$H^2_{(1)}(\Gamma;\B R)\to H^2(\Gamma;\B R)$
is a codifferential of a function $\phi\colon\Gamma\to\B R$ such that
for each fixed $h\in\Gamma$ we have $\sup_{g\in\Gamma}|\phi(g)-\phi(gh)+\phi(h)|<\infty$. That is, 
$$
\|h\phi-\phi\|_0
:=\sup_{g\in\Gamma}|\phi(gh)-\phi(g)|<\infty.
$$  
We call such functions \emph{abstractly Lipschitz}.
The above argument can be then summed up as follows.
\begin{proposition}
The kernel of the comparison map $H^2_{(1)}(\Gamma;\B R)\to H^2(\Gamma;\B R)$
is isomorphic to the quotient on abstractly Lipschitz functions divided
by the sum of the space of homomorphisms and the space of bounded functions.\qed
\end{proposition}

On the other hand, given abstractly Lipschitz function $\phi$, the map
$\Phi\colon \Gamma \to \ell^{\infty}\Gamma$ given by $\Phi(g) = g\phi - \phi$
is a cocycle. Indeed, we have
\begin{align*}
\delta \Phi(g,h) &= g\Phi(h) - \Phi(gh) + \Phi(g)\\
&= g\left( h\phi-\phi \right) - gh\phi + \phi + g\phi -\phi\\
&= gh\phi - g\phi - gh\phi +g\phi = 0.
\end{align*}
Taking the class of $\Phi$ in $H^1(\Gamma;\ell^\infty\Gamma/\B R)\cong H^2_{(1)}(\Gamma;\B R)$ 
corresponds to taking class of $\phi$ modulo
homomorphisms (constant cocycles) and bounded functions (coboundaries).

The name {\itshape abstractly Lipschitz} is motivated by the following observation.

\begin{proposition}
Let $\phi$ be a function on a group $\Gamma$.  
The following are equivalent
\begin{enumerate}
\item $\phi$ is abstractly Lipschitz,
\item there exists a length function $|\cdot|$ on $\Gamma$, 
such that $\phi$ is Lipschitz with respect to $|\cdot|$.
\end{enumerate}
\end{proposition}

\begin{proof}
Assume (1). Then $\phi$ is Lipschitz with respect to
$$
|g|_\phi=\|g\phi-\phi\|_0.
$$

Assume (2).  Then $\|g\phi-\phi\|_0\leq |g|$.
\end{proof}
Notice that if $\Gamma$ is finitely generated then abstractly Lipschitz
functions are Lipschitz with respect to the word norm associated
with a finite generating set.

The second construction goes as follows.
Let $\omega\in C^2_{(1)}(\Gamma;\B R)$ 
be an arbitrary 1-bounded 2-cocycle.  For a function $\psi\in\ell^\infty\Gamma$
define $|\psi|_\R{osc}=\sup_{h,h'\in\Gamma}|\psi(h)-\psi(h')|$ and define $\|g\|_\omega=|\omega(g,\cdot)|_\R{osc}$.

\begin{proposition}
The function $\|\cdot\|_\omega$ is a well defined metric on $\Gamma$.  Moreover, if $\eta$ is a bounded 1-cochain,
then $\|\cdot\|_\omega$ and $\|\cdot\|_{\omega+\delta\eta}$ are within finite (at most $2|\eta|_\R{osc}$) distance.
\end{proposition}

\begin{proof}
By cocycle identity
\begin{align*}
\omega(g_1g_2,h)-\omega(g_1g_2,h')
&=\omega(g_2,h)+\omega(g_1,g_2h)-\omega(g_1,g_2)\\
&\hphantom{=}-\omega(g_2,h')-\omega(g_1,g_2h')+\omega(g_1,g_2)\\
&=\omega(g_2,h)-\omega(g_2,h')+\omega(g_1,g_2h)-\omega(g_1,g_2h').
\end{align*}
Taking supremum of the absolute value of the above over all $h$ and $h'$ we get
the triangle identity.

To show that $\|\cdot\|_\omega$ is symmetric first write
$$0=\delta\omega(1,1,h)=\omega(1,h)-\omega(1,h)+\omega(1,h)-\omega(1,1),$$
thus $\omega(1,h)$ does not depend oh $h$.  Then
$$0=\delta\omega(g^{-1},g,h)=\omega(g,h)-\omega(1,h)+\omega(g^{-1},gh)-\omega(g^{-1},g),$$
hence
$$\omega(g,h)+\omega(g^{-1},gh)=\omega(1,1)+\omega(g^{-1},g)=\omega(g,k)+\omega(g^{-1},gk),$$
or, equivalently
$$\omega(g,h)-\omega(g,k)=\omega(g^{-1},gk)-\omega(g^{-1},gh).$$
Taking the supremum on both sides we derive $\|g\|_\omega=\|g^{-1}\|_\omega$.

The last statement is obvious.
\end{proof}

One can rephrase the above statement by saying that every element of $H^2_{(1)}(\Gamma,\B R)$ defines
a class of metrics on $\Gamma$ which are within finite distance from each other.

In \cite[Theorem 5.2]{zbMATH06065373} we essentially prove that a certain
1-bounded 2-cocycle $\F G$ defines a metric on the group of Hamiltonian
diffeomorphisms $\operatorname{Ham}(X,\sigma)$ of a symplectically hyperbolic
manifold $(X,\sigma)$ and prove that each element of this group is undistorted
with respect to $\|\cdot\|_{\F G}$.  Notice, that in \cite{zbMATH06065373} the
function $\|\cdot\|_{\F G}$ was defined slightly differently and did not
satisfy triangle inequality (cf. \cite[Lemma 4.1]{zbMATH06065373}).

Let us compare the two constructions.  Given a 1-bounded 2-cocycle in the
kernel of the comparison map, we defined two norms $\|\cdot\|_{\delta\phi}$
and $|\cdot|_{\phi}$.  Since

\begin{align*}
\delta\phi(g,h)-\delta\phi(g,h')
&=\phi(g)-\phi(gh)+\phi(h)-\phi(g)+\phi(gh')-\phi(h')\\
&=(g\phi-\phi)(h')-(g\phi-\phi)(h),
\end{align*}
we see that $\|g\|_{\delta\phi}=|g\phi-\phi|_\R{osc}\leq2|g|_\phi$.


\section{Hyperbolic classes}\label{S:hyp classes}

Let $p\colon\widetilde X\to X$ be the universal cover of a CW-complex $X$ and let $\Gamma=\pi_1X$.
Let $C^*(\widetilde X)$ and $C^*(X)$ denote CW-cochains on $\widetilde X$ and $X$, respectively. 
The cochains are considered with trivial real coefficients.
We say
that a cochain $\alpha\in C^n(\widetilde X)$ is \emph{tamed} by $\nu\in C^n(\widetilde X)$  
if for every $g\in\Gamma$ and every cell $\triangle$ in $\widetilde X$
$$
|(g\alpha-\alpha)(\triangle)|\leq \nu(\triangle).
$$
Clearly, a cochain is equivariant (it is a lift of a cochain from $C^*(X)$)
if and only if it is tamed by a cochain vanishing identically.

We say that a cocycle $\alpha\in C^n(X)$ is \emph{hyperbolic} if $p^*\alpha$ can be expressed
as  $\delta\beta$ with $\beta$ tamed by some cochain lifted from $X$.
A class in $H^n(X;\B R)$ is \emph{hyperbolic} if it has a hyperbolic representative.

The notion of a hyperbolic class is inspired by Gromov \cite{zbMATH04186563}
who considered differential forms $\alpha$ with the property that the
lift $p^*\alpha$ to the universal cover has a primitive that is bounded with
respect to the induced Riemannian metric (see also the recent paper
of Ascari and Milizia \cite{zbMATH07862377} for new developments). 
By integrating such classes over
cells we get that they yield hyperbolic classes in our sense.
Indeed, if the $n$-skeleton of $X$ is finite, then any homogeneous cochain takes
only finitely many values. Thus, $a\in H^n(X;\B R)$ is hyperbolic if one can
choose $\beta\in C^{n-1}(\widetilde X;\B R)$ with $[\delta\beta]=p^*a$, and
$M\in\B R$, such that for any cell $\triangle$ in $\widetilde X$, $|\beta(\triangle)|<M$.

Hyperbolicity as defined above depends on the cellular structure if $X$ is not compact.
We need it so we can speak of hyperbolic classes in $\R B\Gamma$.
Clearly, hyperbolicity is preserved by cellular maps.

\begin{proposition}\label{P:subdivision}Hyperbolicity does not
change under subdivisions of a given CW-structure.
\end{proposition}

\begin{proof}
Let  $X'$ be a subdivision of $X$.  Wherever we write that a cell in $X'$ divides a cell in $X$
we implicitly assume that they are of the same dimension.

Maps inducing the natural isomorphism on $H^*(X;\B R)$ and $H^*(X';\B R)$ are induced by
$q\colon C^*(X;\B R)\to C^*(X';\B R)$ and $b\colon C^*(X';\B R)\to C^*(X;\B R)$ defined
as $q\alpha(\triangle')=\frac1{n_\triangle}\alpha(\triangle)$, where $\triangle$ is the cell
which $\triangle'$ divides and $n_\triangle$ denote the
number of cells in $X'$ dividing $\triangle$,
and $b\alpha'(\triangle)=\sum\alpha'(\triangle')$, where the sum runs over all cells $\triangle'$ in $X'$
dividing $\triangle$.
We denote corresponding maps for $\widetilde X$ and its induced subdivision $\widetilde X'$ by the same letters.
We observe, in particular, that they commute with~$p^*$.

If $\alpha=p^*\delta\beta$ and $\beta$ is tamed by $\nu$ then $q\alpha=p^*\delta q\beta$ and $q\beta$ is tamed by $q\nu$.
Similarly, if $\alpha'=p^*\delta\beta'$ and $\beta'$ is tamed by $\nu'$ then $b\alpha'=p^*\delta b\beta'$ and $b\beta'$ is tamed by $b\nu'$.
%
%
%
%
\end{proof}

The aim of this section is to prove the following connection between
hyperbolic and 1-bounded classes.

\begin{theorem}\label{T:hyperbolic}
Let $X$ be a finite connected CW--complex with $\pi_1(X)=\Gamma$ and let
$c\colon X\to\R B\Gamma$ be the classifying map.  Then a pullback of a
1-bounded class is hyperbolic.  Moreover, if the universal cover of\/ $X$
is $(n-1)$-connected and $a\in H^n(X;\B R)$ is a hyperbolic class,
then there exists a 1-bounded cocycle 
$\omega \in C^n_{(1)}(\R B\Gamma;\B R)$ such that $c^*[\omega]=a$.
\end{theorem}

\begin{remark}
We do not know wheather one can replace the assumption on higher connectivity
of $\widetilde X$ by assuming that $a$ lies in the image of $c$.
\end{remark}

Until the end of this section  we use the homogeneous notation, since
we need to talk about chains on $\R E\Gamma$.

We say that a $(n-1)$-cochain $\lambda$ \emph{controls} an $n$-cochain $\omega$
if for all $h,g_1,\dots,g_n\in\Gamma$
\begin{equation}
|\omega(h,g_1,\dots,g_n)-\omega(1,g_1,\dots,g_n)|<\lambda(g_1,\dots,g_n).
\label{Eq:control}
\end{equation}
By definition, a homogeneous cochain $\omega$ is 1-bounded if it is controlled
by a homogeneous cochain.

It is straightforward to verify that the map 
$H\colon C^n(\R E\Gamma;\B R)\to C^{n-1}(\R E\Gamma;\B R)$
given by the formula $(H\omega)(g_1,\dots,g_n)=\omega(1,g_1,\dots,g_n)$ is a homotopy
satisfying $\delta H\omega+H\delta\omega=\omega$ (see \cite[Page 18]{zbMATH03935317}).

\begin{lemma}\label{L:controlled-tamed}
Assume that $\omega\in C^n(\R E\Gamma;\B R)$ is closed, $\Gamma$-invariant
and controlled by $\lambda$.  Then $H\omega$ is tamed by $\lambda$.
\end{lemma}

\begin{proof}
\begin{align*}
|(gH\omega-H\omega)(g_1\dots,g_n)|
&=|H\omega(g_1g,\dots,g_ng)-H\omega(g_1,\dots,g_n)|\\
&=|\omega(1,g_1g,\dots,g_ng)-\omega(1,g_1,\dots,g_n)|\\
&=|\omega(g^{-1},g_1,\dots,g_n)-\omega(1,g_1,\dots,g_n)|\\
&\leq\lambda(g_1,\dots,g_n).
\end{align*}
\end{proof}

Given $\eta\in C^{n-1}(\R E\Gamma)$ define $\bar\eta\in C^{n-1}(\R E\Gamma)^\Gamma$
by the formula
$$
\bar\eta(g_1,\dots,g_n)=\eta(1,g_2g_1^{-1},\dots,g_ng_1^{-1}).
$$

\begin{lemma}\label{L:tamed-controlled}
Assume that $\eta\in C^{n-1}(\R E\Gamma;\B R)$ is tamed by $\nu$
and $\delta\eta$ is $\Gamma$-invariant
then $\delta(\eta-\bar\eta)$ is controlled by $\nu$.
\end{lemma}

\begin{proof}
Since $\delta(\eta-\bar\eta)$ is $\Gamma$-invariant,
\begin{align*}
\delta(\eta-\bar\eta)(h,g_1,\dots,g_n)
eta&=\delta(\eta-\bar\eta)(1,g_1h^{-1},\dots,g_nh^{-1})\\
&=\eta(g_1h^{-1},\dots,g_nh^{-1})-\eta(1,g_2h^{-1},\dots,g_nh^{-1})\pm\cdots\\
&\hphantom{=}-\eta(1,g_2g_1^{-1}\dots,g_ng_1^{-1})+\eta(1,g_2h^{-1},\dots,g_nh^{-1})\mp\cdots\\
&=\eta(g_1h^{-1},\dots,g_nh^{-1})-\eta(1,g_2g_1^{-1},\dots,g_ng_1^{-1})\\
&=(h^{-1}\eta-g_1^{-1}\eta)(g_1,\dots,g_n).
\end{align*}
Thus, since, by assumption, $\eta$ is tamed by $\nu$,
\begin{align*}
|\delta(\eta-\bar\eta)(h,g_1,\dots,g_n)-\delta(\eta-\bar\eta)(1,g_1,\dots,g_n)|
&=|(h^{-1}\eta-\eta)(g_1,\dots,g_n)|\\
&\leq\nu(g_1,\dots,g_n).
\end{align*}
\end{proof}

\begin{corollary}\label{C:hyperbolic-1-bounded}
Let $a\in H^n(\R B\Gamma;\B R)$.  Then $a$ is hyperbolic if and only if $a$ has
1-bounded representative.
\end{corollary}

\begin{proof}[Proof of Theorem \ref{T:hyperbolic}]
Let $\omega$ be a 1-bounded cocycle.  Then $[\omega]$ is hyperbolic,
so, by the above corollary, is $c^*[\omega]$.

Let $a\in H^n(X;\B R)$ be hyperbolic.  Taking $\widetilde Y$,
an $\Gamma$-invariant subdivision of the $k$-skeleton of the simplicial
complex $\R E\Gamma$, we may build an equivariant cellular map $\widetilde f\colon\widetilde Y\to\widetilde X$.
Let $f\colon Y=\widetilde Y/\Gamma\to X$ be the corresponding map of the quotients.

Let $\alpha \in C^k(X;\B R)$ be a hyperbolic cocycle and let
$p^*(\alpha)=\delta\beta$, where $\beta$ is a tamed cochain.
Define $\eta=b(\widetilde f^*\beta)$.  By Proposition \ref{P:subdivision},
$\eta$ is tamed.  By Lemma \ref{L:tamed-controlled} $\delta(\eta-\overline\eta)$ is controlled.

The equivariant cocycle $\delta(\eta-\overline\eta)$ is a lift of a cocycle 
$\omega\in C^k(\R B\Gamma;\B R)$.
Being controlled for $\delta(\eta-\overline\eta)$ translates to 
$1$-boudedness of $\omega$; see comment after \eqref{Eq:control}.
Since $\overline\eta$ is a lift of a cocycle from $\R B\Gamma$ we see that
$[\omega]=f^*[\alpha]$.
Since $c\circ f$ is homotopic to the inclusion $Y\to \R B\Gamma$,
it follows that $c^*[\omega]=[\alpha]$
which finishes the proof.
\end{proof}


\section{Hyperbolic classes and amenability}

It immediately follows from Theorems \ref{T:comparison} and \ref{T:hyperbolic} that
any hyperbolic class on a CW--complex $X$ as above is trivial if $\pi_1(X)$ is
amenable. However, a more direct proof yields the following slightly stronger
observation originally due to Brunnbauer Kotschick \cite[Theorem
3.2]{arxiv.0808.1482} who used an isoperimetric characterisation of manifolds
with amenable fundamental group due to Brooks.
A different argument was used in \cite{zbMATH05597817} to show that the
fundamental group of a closed symplectically hyperbolic manifold cannot be
amenable.

\begin{proposition}\label{P:BK}
Let $X$ be a finite complex and let $\alpha$ be a hyperbolic cocycle on $X$.
If the fundamental group of $X$ is amenable, then $\alpha$ is a coboundary.
\end{proposition}

\begin{proof}
Let $p^*\alpha=\delta \beta$ with $\beta$ bounded cochain on $\widetilde X$.
By averaging this equation with respect to a mean on the deck transformation
group $\pi_1(X)$ we find $\beta$ which is $\pi_1(X)$--invariant, thus
$\beta$ is a lift of a cocycle $\beta'$ on $X$ and $\alpha=\delta\beta'$
which proves the claim.
\end{proof}


\section{Test case: Thompson's group $F$}\label{S:test}
It is a stimulating open question to determine whether Thompson's group $F$ is
amenable. Since it is a split extension 
$$
[F,F]\to F\to \B Z^2,
$$
its amenability is equivalent to the amenability of its commutator subgroup.
Notice that the commutator subgroup $[F,F]$ is boundedly acyclic due to
a recent result of Campagnolo, Fournier-Facio, Lodha and Moraschini
\cite[Corollary 1.14]{2311.16259}.
The cohomology ring $H^*([F,F];\B R)$ is isomorphic to a polynomial ring $\B R[u]$, where $\deg(u)=2$.
The explicit cocycle $\alpha$ (defined in \eqref{Eq:g-s} below) 
representing $u$ has been given by Ghys and Sergiescu
\cite[Corollaire~4.5]{zbMATH04056610}.
Let $\B Z^2\to [F,F]$ be an injective
homomorphism generated by two commuting nontrivial elements.
It is easy to choose such elements so that the pull-back of $u$ is nontrivial
in $H^2(\B Z^2;\B R)$ and, since $\B Z^2$ is amenable, it can't be represented
by a 1-bounded cocycle. This is explained in the first step of the proof of the
following general observation
which provides further evidence for a possible amenability of $F$.

Let $\alpha \colon F'\times F'\to \B R$ be a cocycle defined by the
following formula
\begin{equation}
\alpha(f,g) = \sum_x \det
\left( 
\begin{smallmatrix}
log_2\ g'_L(x)\ &\ log_2\ (f\circ g)'_L(x)\\
log_2\ g'_R(x)\ &\ log_2\ (f\circ g)'_R(x)\\
\end{smallmatrix}
 \right),
\label{Eq:g-s}
\end{equation}
where the summation is over the breakpoints. The subscripts $L$ and $R$
refer to the left and right derivatives. This cocycle represents
a generator of the cohomology ring $H^*(F';\B R)=\B R[u]$.

\begin{proposition}\label{P:thompson}
For every positive integer $n\in \B N$ there exists an injective homomorphism
$\psi\colon \B Z^{2n}\to [F,F]$ such that $\psi^*(u^n)\neq 0$
in $H^{2n}(\B Z^{2n};\B R)$. Consequently, $u^n$ cannot be represented
by a 1-bounded cocycle.
\end{proposition}

\begin{proof}
Let $F'=[F,F]$ denote the commutator subgroup of Thompson's group~$F$
and let $f,g\in F'$ be elements defined by the following pictures.
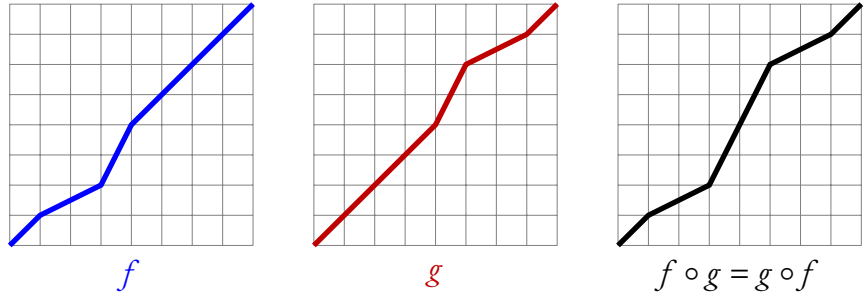
\begin{center}
\begin{tikzpicture}[line width=2pt, scale=0.4]
\draw[line width=.1pt, color=gray] (0,0) grid (8,8);
\draw[blue] (0,0)--(1,1)--(3,2)--(4,4)--(8,8);
\draw[blue](4,-1) node {$f$}; 
\draw[line width=.1pt, color=gray, xshift=10cm] (0,0) grid (8,8);
\draw[red,xshift=10cm] (0,0)--(4,4)--(5,6)--(7,7)--(8,8);
\draw[red,xshift=10cm](4,-1) node {$g$}; 
\draw[line width=.1pt, color=gray, xshift=20cm] (0,0) grid (8,8);
\draw[xshift=20cm] (0,0)--(1,1)--(3,2)--(4,4)--(5,6)--(7,7)--(8,8);
\draw[xshift=20cm](4,-1) node {$f\circ g=g\circ f$}; 
\end{tikzpicture}
\end{center}
Let $\psi\colon \B Z^2\to F'$ be an injective homomorphism defined by 
$$
\psi(e_1)=f\text { and } \psi(e_2)=g,
$$
where $e_1=(1,0)$ and $e_2=(0,1)$.  It is a direct computation that
$\alpha(f,g)=1$ and that $\alpha(g,f)=-1$ (only the breaking point $(1/2,1/2)$
contributes). Since $\zeta_2 = (e_1|e_2)-(e_2|e_1)$ is a cycle representing
nontrivial class in $H_2(\B Z^2;\B R)$ we get that 
$\langle \psi^*\alpha,\zeta_2\rangle=2$ and hence $\psi^*(u)\neq 0$ in $H^2(\B Z^2;\B R)$.
Since $\B Z^2$ is amenable, $u$ cannot be represented by a 1-bounded cocycle.
We use here the bar notation for non-homogeneous chains 
(see Brown \cite[Section II.3]{zbMATH03935317}).

Consider $\psi\colon \B Z^{2n}\to F'$ generated by
$$
f_1,g_1,\ldots,f_n,g_n,
$$
where $f_i$ and $g_i$ are supported on $\left[ \frac{i-1}{n},\frac{i}{n} \right]$
and are rescaled copies of $f$ and $g$ discussed above. Observe that
\begin{align*}
\psi^*\alpha(f_i,g_i) &= 2\\
\psi^*\alpha(f_i,g_j) &= 
\psi^*\alpha(f_i,f_j) = 
\psi^*\alpha(g_i,g_j) = 0 &\text{ for $i\neq j$}. 
\end{align*}
It follows that cocycle $\psi^*\alpha\colon \B Z^{2n}\times \B Z^{2n}\to \B R$ is
a symplectic form and hence its $n$-th power represents a non-trivial
cohomology class in $H^{2n}(\B Z^{2n};\B R)$. That is, $[\psi^*\alpha]^n=
\psi^*[\alpha]^n=\psi^*(u^n)\neq 0$. Since $\B Z^{2n}$ is amenable, $u^n$ cannot
be represented by a $1$-bounded cocycle.
\end{proof}

\subsection*{Acknowledgements}
The first author was partially supported
by Polish \MakeLowercase{\textsc{NCN}} grant 
\MakeLowercase{\textsc{2017/27/B/ST1/01467}}.
This work of the second author was funded by Leverhulme Trust Research Project
Grant \MakeLowercase{\textsc{RPG-2017-159}}.

The authors would like to thank Collin Bleak for discussions about Thompson's
group $F$ and Francesco Fournier-Facio for bringing their attention to the
literature on weakly bounded cocycles.

\bibliography{semibounded}
\bibliographystyle{acm}

\bigskip
\newcommand{\email}[1]{\href{mailto:#1}{{\color{red}\tt #1}}}
\scriptsize
\begin{tabbing}
\hspace{.5\hsize}\=\hspace{.5\hsize}\kill
\textsc{Światosław R.~Gal}:\>\textsc{Jarek Kędra}:\\
Instytut Matematyczny\>Mathematical Sciences\\
Uniwersytet Wrocławski\>University of Aberdeen\\
pl. Grunwaldzki 2\>Fraser Noble Building\\
50-384 Wrocław\>Aberdeen \MakeLowercase{\textsc{AB243UE}}\\
Poland\>Scotland, \MakeLowercase{\textsc{UK}}\\
\email{sgal@mimuw.edu.pl}\>\email{kedra@abdn.ac.uk}
\end{tabbing}

\end{document}